\documentclass[12pt]{amsart}
\usepackage{ifthen,verbatim}
\usepackage{floatflt,graphicx,graphics}

 \usepackage{tikz} 
\usepackage{mathrsfs}
\usepackage{amsmath,amssymb,amsthm}

\numberwithin{equation}{section}
\nonstopmode
\numberwithin{equation}{section}
\theoremstyle{plain}
\newtheorem{theorem}[equation]{Theorem}
\newtheorem{conjecture}[equation]{Conjecture}
\newtheorem{lemma}[equation]{Lemma}
\newtheorem{corollary}[equation]{Corollary}

\newtheorem{proposition}[equation]{Proposition}

\theoremstyle{definition}
\newtheorem{definition}[equation]{Definition}

\newtheorem{remark}[equation]{Remark}

\theoremstyle{remark}

\newcommand{\R}{\mathbb{R}}

\newcommand{\C}{\mathbb{C}}

\newcommand{\B}{\mathbb{B}}

\newcommand{\uhp}{\mathbb{H}}

{\qed\bigskip}

\newcounter{alphabet}

\newcounter{minutes}
\divide\time by 60
\newcounter{hours}
\multiply\time by 60
\addtocounter{minutes}{-\time}

\begin{document}
\bibliographystyle{amsplain}
\title[Triangular Ratio Metric in the Unit Disk]
{
Triangular Ratio Metric in the Unit Disk
}

\def\thefootnote{}
\footnotetext{
\texttt{\tiny File:~\jobname .tex,
          printed: \number\year-\number\month-\number\day,
          \thehours.\ifnum\theminutes<10{0}\fi\theminutes}
}
\makeatletter\def\thefootnote{\@arabic\c@footnote}\makeatother

\author[O. Rainio]{Oona Rainio}
\address{Department of Mathematics and Statistics, University of Turku, FI-20014 Turku, Finland}
\email{ormrai@utu.fi}
\author[M. Vuorinen]{Matti Vuorinen}
\address{Department of Mathematics and Statistics, University of Turku, FI-20014 Turku, Finland}
\email{vuorinen@utu.fi}

\begin{abstract}
The triangular ratio metric is studied in a domain $G\subsetneq\R^n$, $n\geq2$. Several sharp bounds are proven for this metric, especially, in the case where the domain is the unit disk of the complex plane. The results are applied to study the H\"older continuity of quasiconformal mappings.
\end{abstract}
\maketitle

\section{Introduction}

In geometric function theory, metrics are often used to define new types of geometries of subdomains of the Euclidean, Hilbert, Banach and other metric spaces \cite{gh,gmp,hkst,v99}. One can introduce a metric topology and build new types of geometries of a domain $G\subset\R^n$, $n\geq2$, based on metrics. Since the local behavior of functions defined on $G$ is an important area of study, it is natural to require that, given a point in $G$, a metric recognizes points close to it from the boundary $\partial G$. 

Thus, certain constraints on metrics are necessary. A natural requirement
is that the distance defined by a metric for given two points $x,y\in G$ takes into account both how far the points are from each other and also their location with respect to the boundary. Indeed, we require that the closures of the balls defined by the metrics do not intersect the boundary $\partial G$ of the domain. We call these type of metrics \emph{intrinsic metrics}. A generic example of an intrinsic metric is the \emph{hyperbolic metric} \cite{bm} of a planar domain, or its generalization, the \emph{quasihyperbolic metric} \cite{gp} defined in all proper subdomains of $G\subsetneq\R^n$, $n\geq2$.

In the recent years, new kinds of intrinsic geometries have been introduced by numerous authors, see \cite[pp. 18-19]{hkvbook}. A. Papadopoulos lists in \cite[pp. 42-48]{p14} twelve metrics recurrent in function theory. Because there are differences how these metrics catch certain intricate features of functions, using several metrics is often imperative. We might further specify the properties of the metrics by requiring that the intrinsic metric should be compatible with the function classes studied. For instance, some kind of a quasi-invariance property is often valuable. Recall that the hyperbolic metric of a planar domain $G$ is invariant under conformal automorphisms of $G$.
 
In 2002, P. H\"ast\"o \cite{h} introduced the \emph{triangular ratio metric}, defined in a domain $G\subsetneq\R^n$ as the function $s_G:G\times G\to[0,1]$,
\begin{align*}
s_G(x,y)=\frac{|x-y|}{\inf_{z\in\partial G}(|x-z|+|z-y|)}. 
\end{align*}
This metric was studied recently in \cite{chkv,fhmv,hvz}, and our goal here is to continue this investigation. We introduce new methods for estimating the triangular ratio metric in terms of several other metrics and establish several results with sharp constants.

In order to compute the value of the triangular ratio metric between points $x$ and $y$ in a domain $G$, we must find a point $z$ on the boundary of $G$ that gives the infimum for the sum $|x-z|+|z-y|$. This is a very simple task if the domain is, for instance, a half-plane or a polygon, but solving the triangular ratio distance in the unit disk is a complicated problem with a very long history, see \cite{fhmv}. However, there are two special cases where this problem becomes trivial: If the points $x$ and $y$ in the unit disk are collinear with the origin or at the same distance from the origin, there are explicit formulas for the triangular ratio metric.

Since the points $x$ and $y$ can be always rotated around their midpoint to end up into one of these two special cases, we can estimate the value of the triangular ratio metric, regardless of how the original points are located in the unit disk. This rotation can be done either by using Euclidean or hyperbolic geometry, and the main result of this article is to prove that both these ways give lower and upper limits for the value of the triangular ratio metric. Note that while we study the midpoint rotation only in the two-dimensional disk, our results can be directly extended into the case $\B^n$, $n\geq3$, for the point $z$ giving the infimum is always on the same two-dimensional disk as $x$, $y$, and the origin. 

The structure of this article is as follows. First, we show a few simple ways to find bounds for the triangular ratio metric in Section \ref{section_3}. We define the Euclidean midpoint rotation and prove the inequalities related to it in Section \ref{section_emr} and then do the same for the hyperbolic midpoint rotation in Section \ref{section_hmr}. Finally, in Section \ref{section_holder}, we explain how finding better bounds for the triangular ratio metric can be useful for studying $K$-quasiconformal mappings in the unit disk.

{\bf Acknowledgements.} The authors are indebted to Professor Masayo Fujimura and Professor Marcelina Mocanu for their kind help in connection with the proof of Theorem  \ref{fujimuras_result}. The research of the first author was supported by Finnish Concordia Fund.

\section{Preliminaries}\label{section_pre}

Let $G$ be some non-empty, open, proper and connected subset of $\R^n$. For all $x\in G$, $d_G(x)$ is the Euclidean distance $d(x,\partial G)=\inf\{|x-z|\text{ }|\text{ }z\in\partial G\}$. Other than the triangular ratio metric defined earlier, we will need the following hyperbolic type metrics:

The \emph{$j^*_G$-metric} $j^*_G:G\times G\to[0,1],$
\begin{align*}
j^*_G(x,y)=\frac{|x-y|}{|x-y|+2\min\{d_G(x),d_G(y)\}},    
\end{align*}
the \emph{point pair function} $p_G:G\times G\to[0,1],$
\begin{align*}
p_G(x,y)=\frac{|x-y|}{\sqrt{|x-y|^2+4d_G(x)d_G(y)}}   
\end{align*}
and the \emph{Barrlund metric} $b_{G,p}:G\times G\to[0,\infty)$,
\begin{align*}
b_{G,p}(x,y)=\sup_{z\in\partial G}\frac{|x-y|}{(|x-z|^p+|z-y|^p)^{1\slash p}}.
\end{align*}
Note that the function $p_G$ is not metric in all domains \cite[Rmk 3.1 p. 689]{chkv}.

The hyperbolic metric is defined as
\begin{align*}
\text{ch}\rho_{\uhp^n}(x,y)&=1+\frac{|x-y|^2}{2d_{\uhp^n}(x)d_{\uhp^n}(y)},\quad x,y\in\uhp^n,\\
\text{sh}^2\frac{\rho_{\B^n}(x,y)}{2}&=\frac{|x-y|^2}{(1-|x|^2)(1-|y|^2)},\quad x,y\in\B^n
\end{align*}
in the upper half-plane $\uhp^n$ and in the Poincaré unit ball $\B^n$, respectively \cite[(4.8), p. 52; (4.14), p. 55]{hkvbook}. In the two-dimensional unit disk,
\begin{align*}
\text{th}\frac{\rho_{\B^2}(x,y)}{2}&=\text{th}(\frac{1}{2}\log(\frac{|1-x\overline{y}|+|x-y|}{|1-x\overline{y}|-|x-y|}))=|\frac{x-y}{1-x\overline{y}}|=\frac{|x-y|}{A[x,y]},
\end{align*}
where $\overline{y}$ is the complex conjugate of $y$ and $A[x,y]=\sqrt{|x-y|^2+(1-|x|^2)(1-|y|^2)}$ is the Ahlfors bracket \cite[(3.17) p. 39]{hkvbook}. The hyperbolic segment between points $x$ and $y$ is denoted by $J[x,y]$, while Euclidean lines, line segments, balls and spheres are written in forms $L(x,y)$, $[x,y]$, $B^n(x,r)$ and $S^{n-1}(x,r)$, respectively, just like in \cite[pp. vii-xi]{hkvbook}. Note that if the center $x$ or the radius $r$ is not specified in the notations $B^n(x,r)$ and $S^{n-1}(x,r)$, it means that $x=0$ and $r=1$. The hyperbolic ball is denoted by $B_\rho^n(q,R)$, as in the following lemma.

\begin{lemma}\label{lem_rhoball}\emph{\cite[(4.20) p. 56]{hkvbook}}
The equality $B_\rho^n(q,R)=B^n(j,h)$ holds, if 
\begin{align*}
j=\frac{q(1-t^2)}{1-|q|^2t^2},\quad
h=\frac{(1-|q|^2)t}{1-|q|^2t^2}\quad\text{and}\quad
t={\rm th}\left(\frac{R}{2}\right).
\end{align*}
\end{lemma}

For the results of Section \ref{section_hmr}, the formula for the hyperbolic midpoint is needed.

\begin{theorem}\cite[Thm 1.4, p.3]{wvz}
For all $x,y\in\B^2$, the hyperbolic midpoint $q$ of $J[x,y]$ with $\rho_{\B^2}(x,q)=\rho_{\B^2}(q,y)=\rho_{\B^2}(x,y)\slash2$ is given by
\begin{align*}
q=\frac{y(1-|x|^2)+x(1-|y|^2)}{1-|x|^2|y|^2+A[x,y]\sqrt{(1-|x|^2)(1-|y|^2)}}.   
\end{align*}
\end{theorem}

Furthermore, the next results will be useful when studying the triangular ratio metric in the unit disk.

\begin{theorem}\label{rhojsp_inB}
\emph{\cite[p. 460]{hkvbook}} For all $x,y\in\B^n$,
\begin{align*}
{\rm th}\frac{\rho_{\B^n}(x,y)}{4}\leq j^*_{\B^n}(x,y)\leq s_{\B^n}(x,y)\leq p_{\B^n}(x,y)\leq{\rm th}\frac{\rho_{\B^n}(x,y)}{2}\leq2{\rm th}\frac{\rho_{\B^n}(x,y)}{4}.    
\end{align*}
\end{theorem}

\begin{theorem}\label{find_zinB}
\emph{\cite[p. 138]{fhmv}} For all $x,y\in\B^n$, the radius drawn to the point $z$ giving the infimum $\inf_{z\in S^{n-1}}(|x-z|+|z-y|)$ bisects the angle $\measuredangle XZY$.
\end{theorem}

\begin{lemma}\label{lem_smetricinB_collinear}
\emph{\cite[11.2.1(1) p. 205]{hkvbook}}
For all $x,y\in\B^n$,
\begin{align*}
s_{\B^n}(x,y)\leq\frac{|x-y|}{2-|x+y|},    
\end{align*}
where the equality holds if the points $x,y$ are collinear with the origin.
\end{lemma}

\begin{theorem}\label{smetricinB_forconjugate}
\emph{\cite[Thm 3.1, p. 276]{hkvz}} If $x=h+ki\in\B^2$ with $h,k>0$, then
\begin{align*}
s_{\B^2}(x,\overline{x})&=|x|\text{ if }|x-\frac{1}{2}|>\frac{1}{2},\\
s_{\B^2}(x,\overline{x})&=\frac{k}{\sqrt{(1-h)^2+k^2}}\leq|x|\text{ otherwise.}
\end{align*}
\end{theorem}

\begin{remark}\label{rmk_onlyonez}
If $x,y\in\B^n$ such that $|x|=|y|$ and there is only one point $z\in S$ giving the infimum $\inf_{z\in S^{n-1}}(|x-z|+|z-y|)$, then it can be verified with Theorem \ref{find_zinB} that $z=(x+y)\slash|x+y|$.
\end{remark}

\section{Bounds for triangular ratio metric}\label{section_3}

In this section, we will introduce a few different upper and lower bounds for the triangular ratio metric in the unit disk $\B^2$, using the Barrlund metric and a special lower limit function. There are numerous similar results already in literature, but we complement them and prove that our inequalities are sharp by showing that they have the best possible constant. First, we introduce the following inequality: 

\begin{lemma}\label{upperlimfors_starlike}
For all $y\in G$, the inequality 
\begin{align*}
s_G(x,y)\leq\frac{|x-y|}{d_G(x)+\sqrt{|x-y|^2+d_G(x)^2-2d_G(x)\sqrt{|x-y|^2-d_G(y)^2}}},    
\end{align*}
holds, if the domain $G$ is starlike with respect to $x\in G$ and $d_G(x)+d_G(y)\leq|x-y|$.
\end{lemma}
\begin{proof}
Let $G$ be starlike with respect to $x\in G$ and consider an arbitrary point $y\in G$. Clearly, $B^n(x,d_G(x)),B^n(y,d_G(y))\subset G$. It also follows from the starlikeness of $G$ that the convex hull
$\cup_{u\in B^n(y,d_G(y))}[x,u]$ must belong to $G$. Fix $u,v\in S^{n-1}(y,d_G(y))$, $u\neq v$ on the same plane with the points $x,y$ so that the lines $L(x,u)$ and $L(y,v)$ are tangents of $S^{n-1}(y,d_G(y))$, and fix $z_1\in S^{n-1}(x,d_G(x))\cap[x,u]$.

By the starlikeness of $G$, $\cup_{s\in B^n(y,d_G(y))}[x,s]\subset G$, so it follows that $z_1$ fulfills
\begin{align*}
|x-z_1|+|z_1-y|\leq\inf_{z\in\partial G}(|x-z|+|z-y|)
\quad\Leftrightarrow\quad
s_G(x,y)\leq\frac{|x-y|}{|x-z_1|+|z_1-y|}.
\end{align*}
Here, $|x-z_1|=d_G(x)$ and, with the information that $|u-y|=|y-v|=d_G(y)$ and $\measuredangle XUY=\measuredangle YVX=\pi\slash2$, we can conclude that 
\begin{align*}
|z_1-y|=\sqrt{|x-y|^2+d_G(x)^2-2d_G(x)\sqrt{|x-y|^2-d_G(y)^2}}. 
\end{align*}
Thus, the lemma follows.
\end{proof}

\begin{remark}
The same method as in the proof of Lemma \ref{upperlimfors_starlike} can be also applied into the case where $G$ is convex. In that case, $J=\cup_{s\in B^n(x,d_G(x)),\text{ }t\in B^n(y,d_G(y))}[s,t]\subset G$ for all $x,y\in G$, so
\begin{align*}
s_G(x,y)\leq\frac{|x-y|}{|x-z_1|+|z_1-y|},
\end{align*}
where $z_1$ is chosen from $\partial J$ so that $|x-z_1|+|z_1-y|$ is at minimum. By finding the value of this sum, we end up with the result $s_G(x,y)\leq p_G(x,y)$, which holds by \cite[Lemma 11.6(1), p. 197]{hkvbook}.
\end{remark}

Let us now focus on the Barrlund metric.

\begin{lemma}\label{lem_ine_b.sG}
\emph{\cite[Thm 3.6 p. 7]{fmv}}
For all $x,y\in G\subsetneq\R^n$,
\begin{align*}
s_G(x,y)\leq b_{G,p}(x,y)\leq2^{1-1\slash p}s_G(x,y).   
\end{align*}
\end{lemma}

\begin{theorem}\label{thm_b2inB2}
\emph{\cite[Thm 3.15 p. 11]{fmv}}
For all $x,y\in\B^2$,
\begin{align*}
b_{\B^2,2}(x,y)=\frac{|x-y|}{\sqrt{2+|x|^2+|y|^2-2|x+y|}}. 
\end{align*}
\end{theorem}

\begin{lemma}\label{lem_sblund_inB}
For all $x,y\in\B^2$,
\begin{align*}
\frac{1}{\sqrt{2}}b_{\B^2,2}(x,y)\leq s_{\B^2}(x,y)\leq b_{\B^2,2}(x,y).    
\end{align*}
Furthermore, this inequality is sharp. 
\end{lemma}
\begin{proof}
The inequality follows from Lemma \ref{lem_ine_b.sG}. Let $x=0$ and $y=k$ with $0<k<1$. By Lemma \ref{lem_smetricinB_collinear} and Theorem \ref{thm_b2inB2},
\begin{align*}
s_{\B^2}(x,y)=\frac{k}{2-k}    
\quad\text{and}\quad
b_{\B^2,2}(x,y)=\frac{k}{\sqrt{2+k^2-2k}},    
\end{align*}
so we will have the following limit values
\begin{align*}
\lim_{k\to0^+}\frac{s_{\B^2}(x,y)}{b_{\B^2,2}(x,y)}
=\lim_{k\to0^+}\left(\frac{\sqrt{2+k^2-2k}}{2-k}\right)
=\frac{1}{\sqrt{2}}
\quad\text{and}\quad
\lim_{k\to1^-}\frac{s_{\B^2}(x,y)}{b_{\B^2,2}(x,y)}
=1. 
\end{align*}
Thus, the sharpness follows.
\end{proof}

Let us next study the connection between the Barrlund metric and two other hyperbolic type metrics that can used to bound the value of the triangular ratio metric in the unit disk, see Theorem \ref{rhojsp_inB}.

\begin{theorem}
For all $x,y\in\B^2$, the sharp inequality
\begin{align*}
\frac{1}{2}b_{\B^2,2}(x,y)\leq j^*_{\B^2}(x,y)\leq b_{\B^2,2}(x,y) \end{align*}
holds.
\end{theorem}
\begin{proof}
The inequality follows from Lemma \ref{lem_sblund_inB}, Theorem \ref{rhojsp_inB} and \cite[Thm 2.9(1), p. 1129]{hvz}. By Theorem \ref{thm_b2inB2},
\begin{align*}
\frac{j^*_{\B^2}(x,y)}{b_{\B^2,2}(x,y)}
=\frac{\sqrt{2+|x|^2+|y|^2-2|x+y|}}{|x-y|+2-|x|-|y|}. 
\end{align*}
For $x=0$ and $y=k$ with $0<k<1$,
\begin{align*}
\lim_{k\to1^-}\frac{j^*_{\B^2}(x,y)}{b_{\B^2,2}(x,y)}
=\lim_{k\to1^-}\left(\frac{\sqrt{2+k^2-2k}}{2}\right)
=\frac{1}{2}. 
\end{align*}
and, for $x=-k$ and $y=k$ with $0<k<1$,
\begin{align*}
\lim_{k\to1^-}\frac{j^*_{\B^2}(x,y)}{b_{\B^2,2}(x,y)}
=\lim_{k\to1^-}\left(\sqrt{\frac{1+k^2}{2}}\right)
=1. 
\end{align*}
Thus, the sharpness follows.
\end{proof}

\begin{theorem}
For all $x,y\in\B^2$, the sharp inequality
\begin{align*}
\frac{1}{\sqrt{2}}b_{\B^2,2}(x,y)\leq p_{\B^2}(x,y)\leq \frac{\sqrt{10}+\sqrt{2}}{4}b_{\B^2,2}(x,y)
\end{align*}
holds.
\end{theorem}
\begin{proof}
Consider now the quotient
\begin{align}\label{pblund_quo}
\frac{p_{\B^2}(x,y)}{b_{\B^2,2}(x,y)}=\sqrt{\frac{2+|x|^2+|y|^2-2|x+y|}{|x-y|^2+4(1-|x|)(1-|y|)}}.    
\end{align}
By Lemma \ref{lem_sblund_inB} and \cite[11.16(1), p. 203]{hkvbook}, $b_{\B^2,2}(x,y)\slash\sqrt{2}\leq s_{\B^2}(x,y)\leq p_{\B^2}(x,y)$ holds for all $x,y\in\B^2$. This inequality is sharp, because, for $x=0$ and $y=k$,
\begin{align*}
\lim_{k\to0^+}\frac{p_{\B^2}(x,y)}{b_{\B^2,2}(x,y)}
=\lim_{k\to0^+}\left(\frac{\sqrt{k^2+2k+2}}{2-k}\right)
=\frac{1}{\sqrt{2}}. 
\end{align*}
Without loss of generality, fix $x=h$ and $y=je^{\mu i}$ with $0\leq h\leq j<1$ and $0<\mu<2\pi$. The quotient \eqref{pblund_quo} is now
\begin{align*}
\frac{p_{\B^2}(x,y)}{b_{\B^2,2}(x,y)}=\sqrt{\frac{2+h^2+j^2-2\sqrt{h^2+j^2+2hj\cos(\mu)}}{h^2+j^2-2hj\cos(\mu)+4(1-h)(1-j)}}.    
\end{align*}
This is decreasing with respect to $\cos(\mu)$, so we can assume that $\mu=\pi$ and $\cos(\mu)=-1$, when looking for the maximum of this quotient. It follows that
\begin{align*}
\frac{p_{\B^2}(x,y)}{b_{\B^2,2}(x,y)}=\sqrt{\frac{(1+h)^2+(1-j)^2}{(h+j)^2+4(1-h)(1-j)}}
=\sqrt{\frac{(1+h)^2+(1-h-q)^2}{(2h+q)^2+4(1-h)(1-h-q)}},
\end{align*}
where $q=j-h\geq0$. The quotient above is clearly decreasing with respect to $q$. Thus, let us fix $j=h$. It follows that
\begin{align*}
\frac{p_{\B^2}(x,y)}{b_{\B^2,2}(x,y)}=\sqrt{\frac{2+2h^2}{8h^2-8h+4}}=\sqrt{\frac{1+h^2}{4h^2-4h+2}}\equiv\sqrt{f(h)},   
\end{align*}
where $f:[0,1)\to\R$, $f(h)=(1+h^2)\slash(4h^2-4h+2)$. By differentiation, for $0\leq h <1$, 
\begin{align*}
f'(h)=\frac{\partial}{\partial h}\left(\frac{1+h^2}{4h^2-4h+2}\right)=\frac{-(h^2+h-1)}{(2h^2-2h+1)^2}=0\quad\Leftrightarrow\quad
h=\frac{\sqrt{5}-1}{2}.
\end{align*}
Since $f(0.1)>1$ and $f(0.9)<0$, the quotient \eqref{pblund_quo} has a maximum $\sqrt{f((\sqrt{5}-1)\slash2)}=(\sqrt{10}+\sqrt{2})\slash4$ and the other part of the theorem follows.
\end{proof}

Finally, we will introduce one special function defined in the punctured unit disk.

\begin{definition}
For $x,y\in\B^2\backslash\{0\}$, define
\begin{align*}
\text{low}(x,y)=\frac{|x-y|}{\min\{|x-y^*|,|x^*-y|\}},  \end{align*}
where $x^*=x\slash|x|^2$ and $y^*=y\slash|y|^2$.
\end{definition}

\begin{remark}
The low-function is not a metric on the punctured unit disk: By choosing points $x=0.3$, $y=-0.1$ and $z=0.1$, we will have
\begin{align*}
0.117\approx{\rm low}(x,y)>{\rm low}(x,z)+{\rm low}(z,y)\approx0.0817,    
\end{align*}
so the triangle inequality does not hold. 

Furthermore, because $A[x,y]=|x||y-x^*|$ for $x,y\in\B^n\setminus\{0\}$, it follows that
\begin{equation} \label{lowandrho}
{\rm th} \frac{\rho_{\B^2}(x,y)}{2}=\frac{|x-y|}{ |x||y- x^*|}\geq{\rm low}(x,y),
\end{equation}
see \cite[7.44(20)]{avv}. Note also that, by \cite[7.42(1)]{avv}, the left hand side of \eqref{lowandrho} defines a metric.
\end{remark}

This low-function is a suitable lower bound for the triangular ratio metric, as the next theorem states.

\begin{lemma}\label{lem_slow}
For all $x,y\in\B^2\backslash\{0\}$, the inequality $s_{\B^2}(x,y)\geq{\rm low}(x,y)$ holds.
\end{lemma}
\begin{proof}
Suppose that $|x-y^*|\leq|x^*-y|$ and fix $z_1\in[x,y^*]\cap S^1$. Clearly, 
\begin{align*}
&d(y,S^1)<d(y^*,S^1)\quad\Leftrightarrow\quad
1-|y|<|y^*|-1=\frac{1}{|y|}-1\\
&\Leftrightarrow\quad|y|-2+\frac{1}{|y|}=\frac{1}{|y|}(|y|-1)^2>0.
\end{align*}
It follows from this that
\begin{align*}
s_{\B^2}(x,y)\geq\frac{|x-y|}{|x-z_1|+|z_1-y|}\geq\frac{|x-y|}{|x-z_1|+|z_1-y^*|}=\frac{|x-y|}{|x-y^*|}=\text{low}(x,y).   
\end{align*}
\end{proof}

As a lower bound for $s_{\B^2}(x,y)$, the low-function is essentially sharp, when $\max\{|x|,|y|\}\to1$. However, the low-function does not give any useful upper limits for the triangular ratio metric, unless we limit from below the absolute value of the points inspected. This can be seen our next theorem.

\begin{theorem}
For all $x,y\in\B^2\backslash\{0\}$, the triangular ratio metric and its lower bound fulfill
\begin{align*}
\sup\{\frac{s_{\B^2}(x,y)}{{\rm low}(x,y)}\text{ }|\text{ }\max\{|x|,|y|\}\geq r\}\leq\frac{1+r}{2r}, 
\end{align*}
where the equality holds if $\max\{|x|,|y|\}=r$.
\end{theorem}
\begin{proof}
Consider the quotient
\begin{align}\label{slow_quo}
\frac{s_{\B^2}(x,y)}{\text{low}(x,y)}=\frac{\min\{|x-y^*|,|x^*-y|\}}{\inf_{z\in S^1}(|x-z|+|z-y|)}.    
\end{align}
Fix $x,y\in\B^2$ such that $0<|x|\leq|y|$ and choose $z\in S^1$ so that it gives the infimum in the denominator of the quotient \eqref{slow_quo}. Let $k_0=\measuredangle ZOX$ and $k_1=\measuredangle ZOY$, where the point $o$ is the origin. Note that, by Theorem \ref{find_zinB}, $\measuredangle XZO=\measuredangle OZY$, so it follows that $0\leq k_1\leq k_0\leq\pi\slash2$. We can write that
\begin{align*}
\inf_{z\in S^1}(|x-z|+|z-y|)=\sqrt{|x|^2+1-2|x|\cos(k_0)}+\sqrt{|y|^2+1-2|y|\cos(k_1)}.
\end{align*}
Furthermore,
\begin{align*}
|x-y^*|=\sqrt{|x|^2+\frac{1}{|y|^2}-2\frac{|x|}{|y|}\cos(k_0+k_1)},\\
|x^*-y|=\sqrt{|y|^2+\frac{1}{|x|^2}-2\frac{|y|}{|x|}\cos(k_0+k_1)}.
\end{align*}

Now, we can find an upper bound for the quotient \eqref{slow_quo}:
\begin{align}
\frac{s_{\B^2}(x,y)}{\text{low}(x,y)}
&\leq\frac{|x-y^*|}{\inf_{z\in S^1}(|x-z|+|z-y|)}\nonumber\\
&\leq\sup_{0\leq k_1\leq k_0\leq\pi\slash2}\frac{\sqrt{|x|^2+1\slash|y|^2-2(|x|\slash|y|)\cos(k_0+k_1)}}{\sqrt{|x|^2+1-2|x|\cos(k_0)}+\sqrt{|y|^2+1-2|y|\cos(k_1)}}\nonumber\\
&=\left(\inf_{0\leq k_1\leq k_0\leq\pi\slash2}\frac{\sqrt{|x|^2+1-2|x|\cos(k_0)}+\sqrt{|y|^2+1-2|y|\cos(k_1)}}{\sqrt{|x|^2+1\slash|y|^2-2(|x|\slash|y|)\cos(k_0+k_1)}}\right)^{-1}\nonumber\\
&\leq\left(\inf_{0\leq k_1\leq 
k_0\leq\pi\slash2}\sqrt{\frac{|x|^2+1-2|x|\cos(k_0)}{|x|^2+1\slash|y|^2-2(|x|\slash|y|)\cos(k_0+k_1)}}\right.\nonumber\\
&\qquad\left.+\inf_{0\leq k_1\leq k_0\leq\pi\slash2}\sqrt{\frac{|y|^2+1-2|y|\cos(k_1)}{|x|^2+1\slash|y|^2-2(|x|\slash|y|)\cos(k_0+k_1)}}\right)^{-1}\nonumber\\
&=\left(
\sqrt{\frac{|x|^2+1-2|x|}{|x|^2+1\slash|y|^2-2(|x|\slash|y|)}}+
\sqrt{\frac{|y|^2+1-2|y|}{|x|^2+1\slash|y|^2-2(|x|\slash|y|)}}
\right)^{-1}\nonumber\\
&=\left(
\frac{1-|x|}{1\slash|y|-|x|}+
\frac{1-|y|}{1\slash|y|-|x|}
\right)^{-1}
=\frac{1\slash|y|-|x|}{2-|x|-|y|}\label{quo_fory}.
\end{align}

Let us yet find another upper bound for the quotient \eqref{quo_fory}. It can be shown by differentiation that the function $f:(0,1)\to\R$, 
\begin{align*}
f(|x|)=\frac{1\slash|y|-|x|}{2-|x|-|y|}    
\end{align*}
is increasing. It follows from this that
\begin{align*}
&|x|\leq|y|\quad\Leftrightarrow\quad f(|x|)\leq f(|y|)\quad\Leftrightarrow\quad \frac{1\slash|y|-|x|}{2-|x|-|y|}\leq\frac{1\slash|y|-|y|}{2-|y|-|y|}=\frac{1+|y|}{2|y|}.
\end{align*}
Thus, for all $x,y\in\B^2$ such that $0<|x|\leq|y|$, the quotient \eqref{slow_quo} fulfills the inequality
\begin{align}\label{ineq_slowquo}
\frac{s_{\B^2}(x,y)}{\text{low}(x,y)}\leq\frac{1\slash|y|-|x|}{2-|x|-|y|}\leq\frac{1+|y|}{2|y|}.
\end{align}

Fix now $x=1\slash2$ and $y=1\slash2+j$ with $0<j<1\slash2$. The quotient \eqref{slow_quo} is now
\begin{align*}
\frac{s_{\B^2}(x,y)}{\text{low}(x,y)}
=\frac{3+2j}{(2+4j)(1-j)}
=\frac{1+|y|}{2|y|(1-j)}
=\frac{1}{1-j}\cdot\frac{1+|y|}{2|y|},
\end{align*}
and it has the limit value
\begin{align*}
\lim_{j\to0^+}\frac{s_{\B^2}(x,y)}{\text{low}(x,y)}
=\frac{1+|y|}{2|y|}. 
\end{align*}
Thus, the inequality \eqref{ineq_slowquo} is sharp and this result proves that
\begin{align*}
\sup\frac{s_{\B^2}(x,y)}{\text{low}(x,y)}=\frac{1+\max\{|x|,|y|\}}{2\max\{|x|,|y|\}}.
\end{align*}
Since the quotient $(1+k)\slash(2k)$ is decreasing for $k\in(0,1)$, the theorem follows.
\end{proof}

The low-function yields a lower limit for also other hyperbolic type metrics.

\begin{lemma}
For all $x,y\in\B^2\backslash\{0\}$, the following inequalities hold and are sharp:\newline 
1. ${\rm low}(x,y)\leq\sqrt{2}j^*_{\B^2}(x,y)$,\newline
2. ${\rm low}(x,y)\leq p_{\B^2}(x,y)$,\newline
3. ${\rm low}(x,y)\leq b_{\B^2,2}(x,y)$.\newline
Furthermore, there is no $c>0$ such that ${\rm low}(x,y)\geq c\cdot d(x,y)$ for all $x,y\in \B^2\backslash\{0\}$, where $d\in\{j^*_{\B^2},p_{\B^2},b_{\B^2,2}\}$.
\end{lemma}
\begin{proof}
The inequalities follow from Theorem \ref{rhojsp_inB}, Lemmas \ref{lem_sblund_inB} and \ref{lem_slow}, and \cite[Thm 2.9(1), p. 1129]{hvz}. Let $0<k<1$. Since
\begin{align*}
\lim_{k\to1^-}\frac{\text{low}(k,ke^{2(1-k)i})}{j^*_{\B^2}(k,ke^{2(1-k)i})}
&=\lim_{k\to1^-}\left(\frac{2k(k\sin(1-k)+1-k)}{\sqrt{k^4+1-2k^2\cos(2(1-k))}}\right)
=\sqrt{2},\\
\lim_{k\to1^-}\frac{\text{low}(k,-k)}{p_{\B^2}(k,-k)}
&=\lim_{k\to1^-}\left(\frac{2k\sqrt{2k^2-2k+1}}{k^2+1}\right)
=1,\\
\lim_{k\to1^-}\frac{\text{low}(k,-k)}{b_{\B^2,2}(k,-k)}
&=\lim_{k\to1^-}\left(\frac{\sqrt{2}k}{\sqrt{k^2+1}}\right)
=1,
\end{align*}
the inequalities are sharp. The latter part of the lemma follows the fact that the limit values above are all 0 if $k\to0^-$ instead.
\end{proof}

\section{Euclidean Midpoint Rotation}\label{section_emr}

In this section, we introduce the Euclidean midpoint rotation. Finding the value of the triangular ratio distance for two points in the unit disk is a trivial problem, if the points are collinear with the origin or at same distance from it, see Lemma \ref{lem_smetricinB_collinear} and Theorem \ref{smetricinB_forconjugate}. Since any two points can always be rotated around their midpoint into one of these two positions, this transformation gives us a simple way to estimate the value of the triangular ratio metric of the original points.   

\begin{definition}\label{def_emr}
\emph{Euclidean midpoint rotation.} Choose distinct points $x,y\in\B^2$. Let $k=(x+y)\slash2$, and $r=|x-k|=|y-k|$. Let $x_0,y_0\in S^1(k,r)$, $x_0\neq y_0$, so that $|x_0|=|y_0|$ and the points $x_0,k,y_0$ are collinear. Fix then $x_1,y_1\in S^1(k,r)$ so that $x_1,k,y_1$ are collinear, $|x_1|=|k|+r$ and $|y_1|=|k|-r$. Note that $x_0,y_0,y_1\in\B^2$ always but $x_1$ is not necessarily in $\B^2$. See Figure \ref{fig1}.
\end{definition}

For all $x,y\in\B^2$, $x\neq y$, such that $x_1\in\B^2$, the inequality
\begin{align*}
s_{\B^2}(x_0,y_0)\leq s_{\B^2}(x,y)\leq s_{\B^2}(x_1,y_1)    
\end{align*}
holds, as we will prove in Theorems \ref{thm_x0y0} and \ref{thm_x1y1}. If $x_1\notin\B^2$, $s_{\B^2}(x_1,y_1)$ is not defined but the first part of this inequality holds. In order to prove this result, let us next introduce a few results needed to find the value of $s_G$-diameter of a closed disk in some domain $G$.

\begin{figure}[ht]
    \centering
    \begin{tikzpicture}[scale=8]
    \draw[thick] (1,0) arc (0:90:1);
    \draw[thick] (0.4,0.4) circle (0.3cm);
    \draw[thick] (0.4,0.4) circle (0.01cm);
    \node[scale=1.3] at (0.39,0.46) {$k$};
    \draw (0,0) circle (0.01cm);
    \node[scale=1.3] at (0,0.05) {$o$};
    \draw (0.612,0.612) circle (0.01cm);
    \node[scale=1.3] at (0.612,0.661) {$x_1$};
    \draw (0.188,0.188) circle (0.01cm);
    \node[scale=1.3] at (0.188,0.241) {$y_1$};
    \draw (0.1879,0.6121) circle (0.01cm);
    \node[scale=1.3] at (0.188,0.662) {$x_0$};
    \draw (0.612,0.188) circle (0.01cm);
    \node[scale=1.3] at (0.615,0.243) {$y_0$};
    \draw (0.5,0.683) circle (0.01cm);
    \node[scale=1.3] at (0.5,0.7328) {$x$};
    \draw (0.3,0.117) circle (0.01cm);
    \node[scale=1.3] at (0.29,0.1672) {$y$};
    \draw[thick, dashed] (0.188,0.612) -- (0.612,0.188);
    \draw[thick, dashed] (0,0) -- (0.6121,0.6121);
    \draw[thick, dashed] (0.5,0.683) -- (0.3,0.117);
    \end{tikzpicture}
    \caption{Euclidean midpoint rotation}
    \label{fig1}
\end{figure}
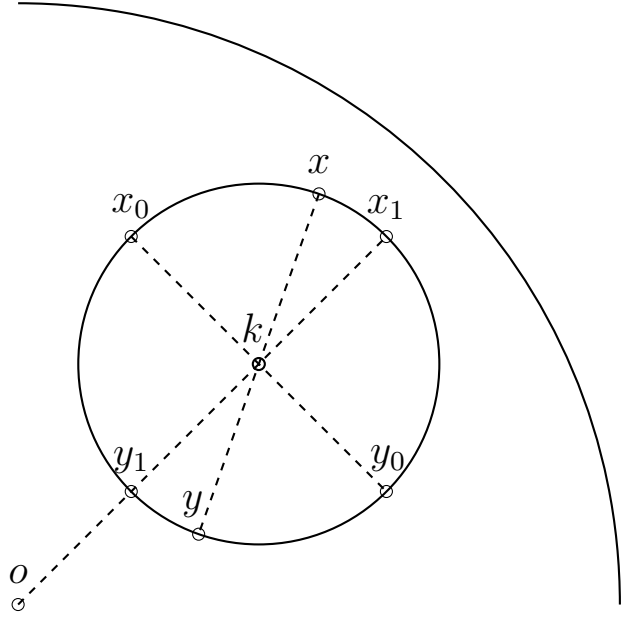

\begin{proposition}\label{prop_linediameter}
For a fixed point $x\in G$ and a fixed direction of $\overrightarrow{xy}$, the value of $s_G(x,y)$ is increasing with respect to $|x-y|$.
\end{proposition}
\begin{proof}
Let $x,y\in G$ and $t\in[x,y]\cap G$. Choose $z\in\partial G$ so that
\begin{align*}
s_G(x,t)=\frac{|x-t|}{|x-z|+|z-t|}.    
\end{align*}
Because the function $f:(0,\infty)\to\R,f(\mu)=(u+\mu)\slash(v+\mu)$ with constants $0<u\leq v$ is increasing,
\begin{align*}
s_G(x,t)\leq\frac{|x-t|+|t-y|}{|x-z|+|z-t|+|t-y|}=\frac{|x-y|}{|x-z|+|z-t|+|t-y|}\leq s_G(x,y).    
\end{align*}
Thus, the result follows.
\end{proof}

\begin{proposition}\label{prop_functionf}
A function $f:[0,\pi\slash2]\to\R$,
\begin{align*}
f(\mu)=\sqrt{u-v\cos(\mu)}+\sqrt{u+v\cos(\mu)},    
\end{align*} 
where $u,v>0$ are constants, is increasing on the interval $\mu\in[0,\pi\slash2]$. 
\end{proposition}
\begin{proof}
Let $s=\cos(\mu)$, so that the function $f$ can be written as $g:[0,1]\to\R$, $g(s)=\sqrt{u-vs}+\sqrt{u+vs}$. By differentiation,
\begin{align*}
g'(s)=\frac{v}{2}\left(\frac{1}{\sqrt{u+vs}}-\frac{1}{\sqrt{u-vs}}\right)\leq0,
\end{align*}
and it follows that the function $g$ is decreasing on the interval $s\in[0,1]$. Because $s=\cos(\mu)$ is decreasing, too, with respect to $\mu$, the function $f$ is increasing.
\end{proof}

\begin{theorem}\label{fujimuras_result}
Fix $j,r,k,z\in\R$ such that $j\leq k<j+r<z$. Choose $x,y\in S^1(j,r)$ so that $\measuredangle ZKX=\mu$ with $0\leq\mu\leq\pi\slash2$ and $k\in[x,y]$. Then the quotient
\begin{align}\label{quo_sxyz}
\frac{|x-y|}{|x-z|+|z-y|} 
\end{align}
is decreasing with respect to $\mu$.
\end{theorem}
\begin{proof}
Suppose without loss of generality that $j=0$ and $r=1$. First, we will consider the special case where $k=0$. From the condition $k\in[x,y]$, it follows that $x,y$ are the endpoints of a diameter of $S^1$ and therefore $|x-y|=2$ for all angles $\mu$. Since $|x|=|y|=1$ and $z=1+d$, we obtain by the law of cosines
\begin{align*}
|x-z|&=\sqrt{1+(1+d)^2-2(1+d)\cos(\mu)},\\
|z-y|&=\sqrt{1+(1+d)^2+2(1+d)\cos(\mu)}.
\end{align*}
The sum $|x-z|+|z-y|$ can be described with the function $f$ of Proposition \ref{prop_functionf} if the constants $u,v$ are replaced with $1+(1+d)^2>0$ and $2(1+d)>0$, respectively. By Proposition \ref{prop_functionf}, this function $f$ is increasing with respect to $\mu\in[0,\pi\slash2]$. Since the quotient \eqref{quo_sxyz} can be clearly written as $2\slash f(\mu)$, it follows that it must be decreasing with respect to $\mu$.  

Suppose now that $S^1(j,r)$ is still the unit circle $S^1$, but let $0<k<1$. The equation of the line $L(x,y)$ can be written as
\begin{align}\label{line_xy_equ}
t+xy\overline{t}=x+y    
\end{align}
with $t\in\C$ as variable. Here, $x$ can be written as $e^{\theta i}$ with $0\leq\theta<\pi\slash2$. Furthermore, the line $L(x,y)$ must contain $k$ and, by substituting $t=k$ in \eqref{line_xy_equ}, we will have 
\begin{align*}
y=\frac{x-k}{kx-1}=\frac{e^{\theta i}-k}{ke^{\theta i}-1}.    
\end{align*}
Consider now a function $h:[0,2\pi)\to\R$,
\begin{align*}
h(\theta)=\frac{|e^{\theta i}-(e^{\theta i}-k)\slash(ke^{\theta i}-1)|}{|e^{\theta i}-z|+|z-(e^{\theta i}-k)\slash(ke^{\theta i}-1)|},  
\end{align*}
which clearly depicts the values of the quotient \eqref{quo_sxyz}. For all $\theta\in[0,\pi\slash2]$, by symmetry,
\begin{align}\label{phi_consequence}
y=e^{\varphi i}=\frac{e^{\theta i}-k}{ke^{\theta i}-1}\quad\Rightarrow\quad h(\theta)=h(-\varphi)   
\end{align}
The function $h$ fulfills $h(0)=h(\pi)=1\slash z$, which is clearly its maximum value. If $\theta=0$, then so is $\mu$, so the maximum of the quotient \eqref{quo_sxyz} is at $\mu=0$. By Rolle's theorem, there is a critical point $\Tilde{\theta}$ such that $f'(\Tilde{\theta})=0$. By the property \eqref{phi_consequence}, $\Tilde{\theta}$ is the solution of
\begin{align*}
 e^{\theta i}=\frac{e^{-\theta i}-k}{ke^{-\theta i}-1}.  
\end{align*}
Thus,
\begin{align*}
\frac{e^{\theta i}+e^{-\theta i}}{2}=k
\quad\Rightarrow\quad
\text{Re}(e^{\theta})=k
\quad\Rightarrow\quad
\mu=\frac{\pi}{2}.
\end{align*}
Consequently, the quotient \eqref{quo_sxyz} attains its minimum value at $\mu=\pi\slash2$. Because there are no other points where the derivative $h'$ is 0 at the open interval $0<\theta<\pi\slash2$ than the one found above, the quotient is monotonic on the interval $\mu\in[0,\pi\slash2]$. To be more specific, the quotient must be decreasing because its maximum is at $\mu=0$ and minimum at $\mu=\pi\slash2$. 

Thus, we have proved that the quotient \eqref{quo_sxyz} is decreasing with respect to $\mu$, regardless of if $k=j$ or $k>j$.  
\end{proof}

\begin{theorem}\label{thm_supInRing}
Fix $S^{n-1}(j,r)\subset\R^n$ and $z\in\R^n$ so that $d=|z-j|-r>0$. Then 
\begin{align*}
\sup_{x,y\in S^{n-1}(j,r)}\frac{|x-y|}{|x-z|+|z-y|}
=\frac{r}{r+d}.
\end{align*}
\end{theorem}
\begin{proof}
Suppose without loss of generality that $n=2$, $j=0$, $r=1$ and $z=d+1\in(1,\infty)$. By symmetry, we can assume that the points $x,y\in S^1$ fulfill $0\leq\arg(x)\leq\pi\slash2$ and $\arg(x)<\arg(y)<2\pi$. We will next prove the theorem by inspecting the quotient \eqref{quo_sxyz} in a few different cases separately.

Consider first the case where $\arg(x)=0$. Now, $x=1$ and $y=e^{\varphi i}$ for some $0<\varphi<2\pi$. It follows that
\begin{align*}
\frac{|x-y|}{|x-z|+|z-y|}=\frac{|1-e^{\varphi i}|}{d+|1+d-e^{\varphi i}|}=\left(\frac{d}{|1-e^{\varphi i}|}+\frac{|1+d-e^{\varphi i}|}{|1-e^{\varphi i}|}\right)^{-1}.
\end{align*}
Since both of the quotients $d\slash|1-e^{\varphi i}|$ and $|1+d-e^{\varphi i}|\slash|1-e^{\varphi i}|$ obtain clearly their minimum with $\varphi=\pi$, the quotient \eqref{quo_sxyz} is at maximum within limitation $x=1$ when $y=-1$.  

Suppose then that $\arg(x)=\theta\neq0$ and $\arg(y)\leq\pi$. Now, we can rotate the points $x,y$ by the angle $\theta$ clockwise about the origin. This transformation does not affect the distance $|x-y|$ but decreases distances $|x-z|$ and $|z-y|$, so it increases the value of the the quotient \eqref{quo_sxyz}. Since $x$ maps into 1 in the rotation, this transformation leads to the first case studied above. 

Finally, consider the case where $\arg(x)\neq0$ and $\pi<\arg(y)<2\pi$. Now, $(x,y)\cap(-1,1)\neq\varnothing$, so we can choose a point $k\in(x,y)\cap(-1,1)$. If $-1<k<0$, we can always reflect the points $x,y$ over the imaginary axis so that the quotient \eqref{quo_sxyz} increases. Thus, we can suppose that $0\leq k<1$. By Theorem \ref{fujimuras_result}, the quotient is decreasing with respect to $\measuredangle ZKX=\mu\in[0,\pi\slash2]$, so its maximum is at $\mu=0$. It follows that $x=1$ and $y=-1$.

Thus, the quotient \eqref{quo_sxyz} obtains its highest value with $x=1$ and $y=-1$. In the general case $x,y\in S^1(j,r)$, this means that $x=j+r$ and $y=j-r$. Since the value of the quotient \eqref{quo_sxyz} is now $r\slash(r+d)$, the result follows.
\end{proof}

\begin{corollary}\label{cor_diskdiameter}
The $s_G$-diameter of a closed ball $J=\overline{B}^n(k,r)$ in a domain $G\subsetneq\R^n$ is $s_G(J)=r\slash(r+d)$, where $d=d(J,\partial G)$.
\end{corollary}
\begin{proof}
Clearly,
\begin{align*}
s_G(J)&=\sup_{x,y\in J}s_G(x,y)=\sup_{x,y\in J}\left(\sup_{z\in\partial G}\frac{|x-y|}{|x-z|+|z-y|}\right)\\
&=\sup_{z\in\partial G}\left(\sup_{x,y\in J}\frac{|x-y|}{|x-z|+|z-y|}\right)
=\sup_{z\in\partial G}\left(\sup_{x,y\in J}s_{\R^n\backslash\{z\}}(x,y)\right)\\
&=\sup_{z\in\partial G}s_{\R^n\backslash\{z\}}(J).
\end{align*}
Trivially, $s_{\R^n\backslash\{z\}}(J)$ is at maximum when the distance $d(z,J)$ is at minimum. Thus,
\begin{align}\label{sup_inJ}
s_G(J)=\sup_{x,y\in J}\frac{|x-y|}{|x-z|+|z-y|},    
\end{align}
where $z\in\partial G$ such that $d=d(z,J)=d(J,\partial G)$. It follows from Proposition \ref{prop_linediameter} that, for all distinct $x,y\in J$, we can choose $s,t\in\partial J$, $s\neq t$, such that $[s,t]=L(x,y)\cap J$ and $s_G(s,t)\geq s_G(x,y)$. Thus, the points $x,y$ giving the supremum in \eqref{sup_inJ} must belong to $S^{n-1}(k,r)$. By Theorem \ref{thm_supInRing}, it follows from this that 
\begin{align*}
s_G(J)=\sup_{x,y\in S^{n-1}(k,r)}\frac{|x-y|}{|x-z|+|z-y|}
=\frac{r}{r+d}.
\end{align*}
\end{proof}

\begin{corollary}\label{cor_diam_diskinB}
The $s_{\B^n}$-diameter of a ball $J=\overline{B}^n(k,r)\subset\B^n$ is $s_{\B^n}(J)=r\slash(1-|k|)$.
\end{corollary}
\begin{proof}
Follows directly from Corollary \ref{cor_diskdiameter}.   
\end{proof}

\begin{corollary}
For all $x,y\in\B^n$ such that $|y|\leq|x|$, the inequality $s_{\B^n}(x,y)\leq|x|$ holds.
\end{corollary}
\begin{proof}
Since $y\in J=B^n(|x|)$, $s_{\B^n}(x,y)\leq s_{\B^n}(J)$ and, by Corollary \ref{cor_diam_diskinB}, $s_{\B^n}(J)=|x|$.
\end{proof}

Consider yet the following situation.

\begin{lemma}
For all points $x\in\B^2\backslash\{0\}$ and $y\in B^2(|x|)$ non-collinear with the origin,
\begin{align*}
s_{\B^2}(x,y)<s_{\B^2}(x,y'),\text{ where }y'=xe^{2\psi i}\text{ and } \psi=\arcsin\left(\frac{|x-y|}{2|x|}\right).    
\end{align*}
\end{lemma}
\begin{proof}
Since $|y'|=|xe^{2\psi i}|=|x|$ and $|x-y'|=|x||1-e^{2\psi i}|=2|x|\sin(\psi)=|x-y|$, the point $y'$ is chosen from $S^1(|x|)\cap S^1(x,|x-y|)$. By symmetry, we can assume that $y'$ is the intersection point closer to $y$. Fix $z$ so that it gives the infimum $\inf_{z\in S^1}(|x-z|+|z-y|)$. If $\mu'=\measuredangle ZXY'$, then $\mu=\measuredangle ZXY=\mu'+\measuredangle Y'XY>\mu'$. Clearly, by the law of cosines,
\begin{align*}
s_{\B^2}(x,y)&=\frac{|x-y|}{|x-z|+|z-y|}
=\frac{|x-y|}{|x-z|+\sqrt{|x-y|^2+|x-z|^2-2|x-y||x-z|\cos(\mu)}}\\
&<\frac{|x-y|}{|x-z|+\sqrt{|x-y|^2+|x-z|^2-2|x-y||x-z|\cos(\mu')}} 
=\frac{|x-y'|}{|x-z|+|z-y'|}\\
&\leq s_{\B^2}(x,y'),
\end{align*}
so the lemma follows.
\end{proof}

Let us now focus on the results related to the Euclidean midpoint rotation.

\begin{proposition}\label{prop_triangles}
Consider two triangles $\triangle YXZ$ and $\triangle Y_0X_0Z_0$ with obtuse angles $\measuredangle YXZ$ and $\measuredangle Y_0X_0Z_0$. Let $k$ and $k_0$ be the midpoints of sides $XY$ and $X_0Y_0$, respectively. Suppose that $|x-y|=|x_0-y_0|$, $|k-z|\leq|k_0-z_0|$ and $\measuredangle ZKX\leq\measuredangle Z_0K_0X_0$. Then,
\begin{align*}
|x-z|+|z-y|\leq |x_0-z_0|+|z_0-y_0|.    
\end{align*}
\end{proposition}
\begin{proof}
Let $r=|x-k|=|x_0-k_0|$, $m=|k-z|$, $m_0=|k_0-z_0|$, $\mu=\measuredangle ZKX$ and $\mu_0=\measuredangle Z_0K_0X_0$, see Figure \ref{fig4}. By the law of cosines,
\begin{align*}
|x-z|+|z-y|&=\sqrt{r^2+m^2-2rm\cos(\mu)}+\sqrt{r^2+m^2+2rm\cos(\mu)},\\
|x_0-z_0|+|z_0-y_0|&=\sqrt{r^2+m_0^2-2rm_0\cos(\mu_0)}+\sqrt{r^2+m_0^2+2rm_0\cos(\mu_0)}.
\end{align*}
Furthermore, by Proposition \ref{prop_functionf}, the function $f:[0,\pi\slash2]\to\R$,
\begin{align*}
f(\mu)=\sqrt{u-v\cos(\mu)}+\sqrt{u+v\cos(\mu)},    
\end{align*}
where $u,v>0$, is increasing with respect to $\mu\in[0,\pi\slash2]$. Note that here $\mu,\mu_0\in[0,\pi\slash2]$ because the triangles already have obtuse angles $\measuredangle YXZ$ and $\measuredangle Y_0X_0Z_0$. Thus, it follows from $\mu\leq\mu_0$ and $m\leq m_0$ that
\begin{align*}
|x-z|+|z-y|&=\sqrt{r^2+m^2-2rm\cos(\mu)}+\sqrt{r^2+m^2+2rm\cos(\mu)}\\
&\leq\sqrt{r^2+m^2-2rm\cos(\mu_0)}+\sqrt{r^2+m^2+2rm\cos(\mu_0)}\\
&\leq\sqrt{r^2+m_0^2-2rm_0\cos(\mu_0)}+\sqrt{r^2+m_0^2+2rm_0\cos(\mu_0)}\\
&=|x_0-z_0|+|z_0-y_0|.
\end{align*}
\end{proof}

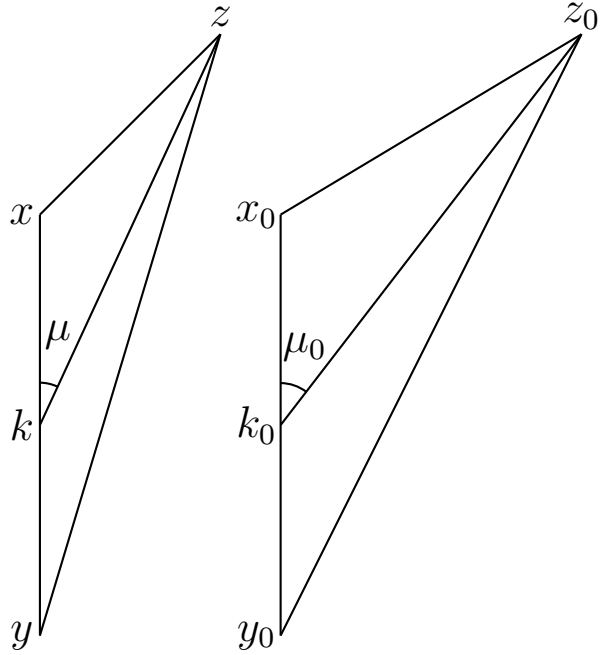
\begin{figure}[ht]
    \centering
    \begin{tikzpicture}[scale=8]
    \draw[thick] (0.6,0) -- (0.6,0.7);
    \draw[thick] (0.6,0.7) -- (0.9,1);
    \draw[thick] (0.6,0) -- (0.9,1);
    \draw[thick] (0.6,0.35) -- (0.9,1);
    \node[scale=1.4] at (0.57,0) {$y$};
    \node[scale=1.4] at (0.57,0.7) {$x$};
    \node[scale=1.4] at (0.9,1.03) {$z$};
    \node[scale=1.4] at (0.57,0.35) {$k$};
    \draw[thick] (0.629,0.414) arc (65:90:0.07);
    \node[scale=1.4] at (0.63,0.5) {$\mu$};
    \draw[thick] (1,0) -- (1,0.7);
    \draw[thick] (1,0.7) -- (1.5,1);
    \draw[thick] (1,0) -- (1.5,1);
    \draw[thick] (1,0.35) -- (1.5,1);
    \node[scale=1.4] at (0.96,0) {$y_0$};
    \node[scale=1.4] at (0.96,0.7) {$x_0$};
    \node[scale=1.4] at (1.5,1.03) {$z_0$};
    \node[scale=1.4] at (0.96,0.35) {$k_0$};
    \draw[thick] (1.043,0.405) arc (52:90:0.07);
    \node[scale=1.4] at (1.04,0.48) {$\mu_0$};
    \end{tikzpicture}
    \caption{The triangles  $\triangle YXZ$ and $\triangle Y_0X_0Z_0$ of Proposition \ref{prop_triangles}}
    \label{fig4}
\end{figure}

\begin{theorem}\label{thm_x0y0}
For all $x,y\in\B^2$,
\begin{align*}
s_{\B^2}(x,y)&\geq s_{\B^2}(x_0,y_0)
\geq\frac{|x-y|}{\sqrt{|x-y|^2+(2-|x+y|)^2}}.
\end{align*}
\end{theorem}
\begin{proof}
Fix $k=(x+y)\slash2$ and $r=|x-k|$. Suppose that $k\neq0$, for otherwise $s_{\B^2}(x,y)=s_{\B^2}(x_0,y_0)$ holds trivially. Without loss of generality, let $0<k<1$ and $\measuredangle XKZ=\nu\in[0,\pi\slash2]$. Now, $\measuredangle YKZ=\pi+\nu$, $x_0=k+ri$ and $y_0=k-ri$. There are two possible cases; either the infimum $\inf_{z_0\in S^1}(|x_0-z_0|+|z_0-y_0|)$ is given by one point $z_0\in S^1$ or there are two possible points $z_0\in S^1$.  

Suppose first that the infimum $\inf_{z_0\in S^1}(|x_0-z_0|+|z_0-y_0|)$ is given by only one point. By Remark \ref{rmk_onlyonez}, this point must be $z_0=1$. Fix $u=r^2+(1-k)^2$ and $v=2r(1-k)$ and consider the function $f$ of Proposition \ref{prop_functionf} for a variable $\nu$. Now, we will have
\begin{align*}
\inf_{z\in S^1}(|x-z|+|z-y|)&\leq|x-1|+|1-y|
=f(\nu)\leq f(\pi\slash2)=|x_0-1|+|1-y_0|\\
&=\inf_{z_0\in S^1}(|x_0-z|+|z-y_0|),  
\end{align*}
from which the inequality $s_{\B^2}(x,y)\geq s_{\B^2}(x_0,y_0)$ follows.

Consider yet the case where there are two points giving the infimum $\inf_{z_0\in S^1}(|x_0-z_0|+|z_0-y_0|)$. By symmetry, we can fix $z_0$ so that $0<\arg(z_0)\leq\pi\slash2$. Now, the infimum $\inf_{z\in S^1}(|x-z|+|z-y|)$ is given by some point $z$ such that $0\leq\arg(z)\leq\arg(z_0)$. If $x,y$ are collinear with the origin, by Lemma \ref{lem_smetricinB_collinear} and Corollary \ref{cor_diam_diskinB},
\begin{align*}
s_{\B^2}(x,y)=\frac{|x-y|}{2-|x+y|}=\frac{r}{1-k}=s_{\B^2}(\overline{B}^n(k,r))\geq s_{\B^2}(x_0,y_0).    
\end{align*}

If $x,y,0$ are non-collinear instead, the triangles $\triangle YXZ$ and $\triangle Y_0X_0Z_0$ exists. The sides $XY$ and $X_0Y_0$ are both the length of $2r$ and have a common midpoint $k$. It follows from Theorem \ref{find_zinB} and the inequality $0<\arg(z)\leq\arg(z_0)\leq\pi\slash2$ that angles $\measuredangle YXZ$ and $\measuredangle Y_0X_0Z_0$ are obtuse, $|k-z|\leq |k-z_0|$ and $\measuredangle ZKX\leq\measuredangle Z_0KX_0$. By Proposition \ref{prop_triangles},
\begin{align*}
|x-z|+|z-y|\leq |x_0-z_0|+|z_0-y_0|,    
\end{align*}
so the inequality $s_{\B^2}(x,y)\geq s_{\B^2}(x_0,y_0)$ follows.

Thus, $s_{\B^2}(x,y)\geq s_{\B^2}(x_0,y_0)$ holds in every cases and, by Theorem \ref{smetricinB_forconjugate},
\begin{align*}
s_{\B^2}(x_0,y_0)\geq\frac{r}{\sqrt{r^2+(1-k)^2}}
=\frac{|x-y|}{\sqrt{|x-y|^2+(2-|x+y|)^2}},
\end{align*}
which proves the latter part of the theorem.
\end{proof}

\begin{theorem}\label{thm_x1y1}
Let $x,y\in\B^2$ with $k=(x+y)\slash2$ and $r=|x-k|$. If $r+k<1$,
\begin{align*}
s_{\B^2}(x,y)\leq s_{\B^2}(x_1,y_1)=\frac{|x-y|}{2-|x+y|}<1. 
\end{align*}
\end{theorem}
\begin{proof}
If $r+k<1$, then $x_1,y_1\in\B^2$ and, by Lemma \ref{lem_smetricinB_collinear},
\begin{align*}
s_{\B^n}(x,y)\leq\frac{|x-y|}{2-|x+y|}
=\frac{|x_1-y_1|}{2-|x_1+y_1|}=s_{\B^2}(x_1,y_1).    
\end{align*}
\end{proof}

\section{Hyperbolic Midpoint Rotation}\label{section_hmr}

In this section, we consider the hyperbolic midpoint rotation. The idea behind it is the same as in the Euclidean midpoint rotation, for our aim is still to rotate the points around their midpoint in order to estimate their triangular ratio distance. However, now the rotation is done by using the hyperbolic geometry of the unit circle instead of the simpler Euclidean method. 

\begin{definition}\label{def_hmr}
\emph{Hyperbolic midpoint rotation.} Choose distinct points $x,y\in\B^2$. Let $q$ be their hyperbolic midpoint and $R=\rho_{\B^2}(x,q)=\rho_{\B^2}(y,q)$. Let $x_2,y_2\in S_\rho^1(q,R)$ so that $|x_2|=|y_2|$ but $x_2\neq y_2$. Fix then $x_3,y_3\in S_\rho^1(q,R)$ so that $x_3,y_3$ are collinear with the origin and $|y_1|<|q|<|x_1|$. See Figure \ref{fig2}.
\end{definition}

The main result of this section is the inequality
\begin{align*}
s_{\B^2}(x_2,y_2)\leq s_{\B^2}(x,y)\leq s_{\B^2}(x_3,y_3).      
\end{align*}
This inequality is well-defined for all distinct $x,y\in\B^2$ because the values of $s_{\B^2}(x_2,y_2)$ and $s_{\B^2}(x_3,y_3)$ are always defined. The first part of the inequality is proved in Theorem \ref{thm_x2y2} and the latter part in Theorem \ref{thm_x3y3}, and the formula for the value of $s_{\B^2}(x_2,y_2)$ is in Theorem \ref{thm_x2y2_value}.
Note that, according to numerical tests, the hyperbolic midpoint rotation gives better estimates for $s_{\B^2}(x,y)$ than the Euclidean midpoint rotation or the point pair function, see Conjecture \ref{conj_hmr}.

\begin{figure}[ht]
    \centering
    \begin{tikzpicture}[scale=8]
    \draw[thick] (1,0) arc (0:90:1);
    \draw[thick] (0.4,0.4) circle (0.3cm);
    \draw (0,0) circle (0.01cm);
    \node[scale=1.3] at (0,0.05) {$o$};
    \draw (0.467,0.467) circle (0.01cm);
    \node[scale=1.3] at (0.51,0.47) {$q$};
    \draw (0.612,0.612) circle (0.01cm);
    \node[scale=1.3] at (0.612,0.661) {$x_3$};
    \draw (0.188,0.188) circle (0.01cm);
    \node[scale=1.3] at (0.188,0.241) {$y_3$};
    \draw[thick, dashed] (0,0) -- (0.6121,0.6121);
    \draw (0.347,0.695) circle (0.01cm);
    \node[scale=1.3] at (0.347,0.73) {$x_2$};
    \draw (0.695,0.347) circle (0.01cm);
    \node[scale=1.3] at (0.74,0.347) {$y_2$};
    \draw[thick, dashed] (0.347,0.695) arc (190:260:0.427);
    \draw (0.5065,0.6795) circle (0.01cm);
    \node[scale=1.3] at (0.5065,0.72) {$x$};
    \draw (0.56,0.145) circle (0.01cm);
    \node[scale=1.3] at (0.565,0.19) {$y$};
    \draw[thick, dashed] (0.5065,0.6795) arc (159.5:212:0.6036);
    \end{tikzpicture}
    \caption{Hyperbolic midpoint rotation}
    \label{fig2}
\end{figure}
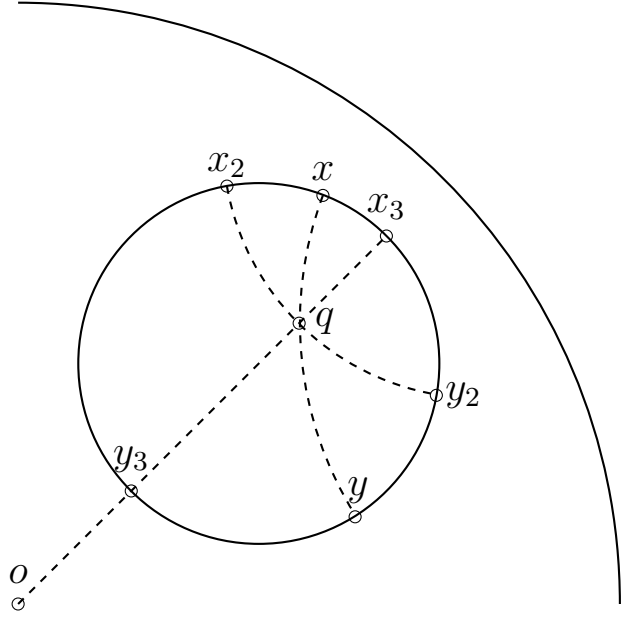

\begin{lemma}\label{lem_x2y2_formulas}
Choose $x,y\in\B^2$ so that their hyperbolic midpoint is $0<q<1$. Let $t={\rm th}(R\slash2)={\rm th}(\rho_{\B^2}(x,y)\slash4)$. Then
\begin{align*}
x_2=\frac{q(1+t^2)}{1+q^2t^2}+\frac{t(1-q^2)}{1+q^2t^2}i\text{ and } y_2=\overline{x_2}.
\end{align*}
\end{lemma}
\begin{proof}
By Lemma \ref{lem_rhoball}, $S_\rho^1(q,R)=S^1(j,h)$ with
\begin{align*}
j=\frac{q(1-t^2)}{1-q^2t^2}\text{ and }
h=\frac{(1-q^2)t}{1-q^2t^2}.
\end{align*}
To find $x_2$ and $y_2$, we need to find the intersection points of $S^1(j,h)$ and $S^1(c,d)$, where $S^1(c,d)\perp S^1$ and $c>1$. Now, $c^2=(q+d)^2=1+d^2$, from which it follows that 
\begin{align*}
d=\frac{1-q^2}{2q}\text{ and }c=\frac{1+q^2}{2q}.    
\end{align*}
Clearly, $x_2=\overline{y_2}$ since both $j,c\in\R$. Let $x_2=u+ri$ and $y_2=u-ri$. Now, $h^2=r^2+(u-j)^2$ and $d^2=r^2+(c-u)^2$. Thus,
\begin{align*}
&h^2-(u-j)^2=d^2-(c-u)^2\quad\Leftrightarrow\quad
h^2-u^2+2ju-j^2=d^2-c^2+2cu-u^2\quad\Leftrightarrow\\
&u=\frac{h^2-j^2-d^2+c^2}{2(c-j)}.
\end{align*}
Since
\begin{align*}
h^2-j^2
&=\frac{(1-q^2)^2t^2}{(1-q^2t^2)^2}-\frac{q^2(1-t^2)^2}{(1-q^2t^2)^2}
=\frac{(t^2-q^2)(1-q^2t^2)}{(1-q^2t^2)^2}
=\frac{t^2-q^2}{1-q^2t^2},\\
-d^2+c^2
&=-\frac{(1-q^2)^2}{4q^2}+\frac{(1+q^2)^2}{4q^2}
=\frac{4q^2}{4q^2}=1,\\
h^2-j^2-d^2+c^2
&=\frac{t^2-q^2}{1-q^2t^2}+1
=\frac{(1+t^2)(1-q^2)}{1-q^2t^2},\\
2(c-j)
&=2(\frac{1+q^2}{2q}-\frac{q(1-t^2)}{1-q^2t^2})
=\frac{(1+q^2)(1-q^2t^2)-2q^2(1-t^2)}{q(1-q^2t^2)}\\
&=\frac{(1-q^2)(1+q^2t^2)}{q(1-q^2t^2)},
\end{align*}
we will have
\begin{align*}
u=\frac{h^2-j^2-d^2+c^2}{2(c-j)}
=\frac{q(1+t^2)(1-q^2)(1-q^2t^2)}{(1-q^2)(1-q^2t^2)(1+q^2t^2)}
=\frac{q(1+t^2)}{1+q^2t^2}.
\end{align*}
From the equality $h^2=r^2+(u-j)^2$, it follows that
\begin{align*}
r&=\sqrt{h^2-(u-j)^2}
=\sqrt{\frac{(1-q^2)^2t^2}{(1-q^2t^2)^2}-\left(\frac{q(1+t^2)}{1+q^2t^2}-\frac{q(1-t^2)}{1-q^2t^2}\right)^2}\\
&=\sqrt{\frac{(1-q^2)^2t^2}{(1-q^2t^2)^2}-\left(\frac{q(1+t^2)(1-q^2t^2)-q(1-t^2)(1+q^2t^2)}{1-q^4t^4}\right)^2}\\
&=\sqrt{\frac{(1-q^2)^2t^2}{(1-q^2t^2)^2}-\left(\frac{2qt^2(1-q^2)}{1-q^4t^4}\right)^2}
=\sqrt{\frac{(1-q^2)^2t^2(1+q^2t^2)^2}{(1-q^4t^4)^2}-\frac{4q^2t^4(1-q^2)^2}{(1-q^4t^4)^2}}\\
&=\sqrt{\frac{t^2(1-q^2)^2(1-q^2t^2)^2}{(1-q^4t^4)^2}}
=\sqrt{\frac{t^2(1-q^2)^2}{(1+q^2t^2)^2}}
=\frac{t(1-q^2)}{1+q^2t^2}.
\end{align*}
\end{proof}

\begin{theorem}\label{thm_x2y2_value}
For all $x,y\in\B^2$ with a hyperbolic midpoint $q\in\B^2\backslash\{0\}$ and $t={\rm th}(\rho_{\B^2}(x,y)\slash4)$,
\begin{align*}
s_{\B^2}(x_2,y_2)&=\sqrt{\frac{|q|^2+t^2}{1+|q|^2t^2}}\text{ if }|q|<t^2,\\
s_{\B^2}(x_2,y_2)&=\frac{t(1+|q|)}{\sqrt{(1+t^2)(1+|q|^2t^2)}}
\leq\sqrt{\frac{|q|^2+t^2}{1+|q|^2t^2}}
\text{ otherwise}.
\end{align*}
\end{theorem}
\begin{proof}
Suppose without loss of generality that $0<q<1$, $x_2=u+ri$ and $y_2=\overline{x_2}$. From Lemma \ref{lem_x2y2_formulas}, it follows that 
\begin{align*}
|x_2|&=\sqrt{u^2+r^2}=\sqrt{\frac{q^2(1+t^2)^2}{(1+q^2t^2)^2}+\frac{t^2(1-q^2)^2}{(1+q^2t^2)^2}}  
=\sqrt{\frac{(q^2+t^2)(1+q^2t^2)}{(1+q^2t^2)^2}}\\
&=\sqrt{\frac{q^2+t^2}{1+q^2t^2}},
\end{align*}
\begin{align*}
&|x_2-\frac{1}{2}|=|u+ri-\frac{1}{2}|>\frac{1}{2}\quad\Leftrightarrow\quad
(u-\frac{1}{2})^2+r^2>\frac{1}{4}\quad\Leftrightarrow\quad
u^2+r^2>u\quad\Leftrightarrow\\
&\frac{q^2(1+t^2)^2}{(1+q^2t^2)^2}+\frac{t^2(1-q^2)^2}{(1+q^2t^2)^2}>\frac{q(1+t^2)}{1+q^2t^2}\quad\Leftrightarrow\\
&q^2(1+t^2)^2+t^2(1-q^2)^2=(t^2+q^2)(1+q^2t^2)>q(1+t^2)(1+q^2t^2)\quad\Leftrightarrow\\
&t^2+q^2>q(1+t^2)\quad\Leftrightarrow\quad
q(1-q)<t^2(1-q)\quad\Leftrightarrow\quad
q<t^2,
\end{align*}
and 
\begin{align*}
(1-u)^2+r^2
&=\left(1-\frac{q(1+t^2)}{1+q^2t^2}\right)^2+\frac{t^2(1-q^2)^2}{(1+q^2t^2)^2}
=\frac{(1-q)^2(1+t^2)(1+q^2t^2)}{(1+q^2t^2)^2}\\
&=\frac{(1-q)^2(1+t^2)}{1+q^2t^2}\quad\Rightarrow\\
\frac{r}{\sqrt{(1-u)^2+r^2}}
&=\frac{t(1-q^2)\sqrt{1+q^2t^2}}{(1-q)(1+q^2t^2)\sqrt{1+t^2}}
=\frac{t(1+q)}{\sqrt{(1+t^2)(1+q^2t^2)}}.
\end{align*}
The result follows now from Theorem \ref{smetricinB_forconjugate}.
\end{proof}

\begin{theorem}\label{gendis_result}\emph{\cite[Prop. 3.1, p. 447]{vw}}
The hyperbolic midpoint of $J[0,b]$ is $[0,b]\cap J[c,d]$ for all $c,d\in S^1$ such that $b\in L(c,d)$ and $c,d$ are non-collinear with the origin.
\end{theorem}

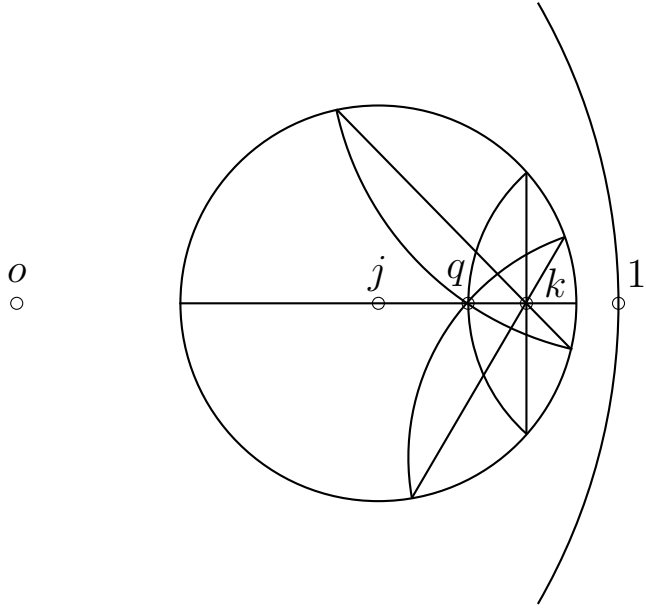
\begin{figure}[ht]
    \centering
    \begin{tikzpicture}[scale=8]
    \draw[thick] (0.866,-0.5) arc (-30:30:1);
    \draw (0,0) circle (0.01cm);
    \node[scale=1.3] at (0,0.05) {$o$};
    \draw (0.75,0) circle (0.01cm);
    \draw[thick] (0.601,0) circle (0.329cm);
    \draw (0.601,0) circle (0.01cm);
    \draw[thick] (0.847,0.217) -- (0.847,-0.217);
    \draw[thick] (0.847,0.217) arc (132:228:0.292);
    \draw (0.847,0) circle (0.01cm);
    \draw (1,0) circle (0.01cm);
    \draw[thick] (0.27,0) -- (0.93,0);
    \node[scale=1.3] at (1.03,0.05) {$1$};
    \draw[thick] (0.531,0.323) -- (0.92,-0.075);
    \draw[thick] (0.531,0.323) arc (192:257:0.521);
    \draw[thick] (0.91,0.11) -- (0.657,-0.323);
    \draw[thick] (0.91,0.11) arc (109.5:190:0.389);
    \node[scale=1.3] at (0.6,0.05) {$j$};
    \node[scale=1.3] at (0.73,0.05) {$q$};
    \node[scale=1.3] at (0.895,0.035) {$k$};
    \end{tikzpicture}
    \caption{Hyperbolic circle $S^1_\rho(q,R)$ with the points $j,q,k$ of Theorem \ref{thm_uvcounter}}
    \label{fig3}
\end{figure}

\begin{theorem}\label{thm_uvcounter}
If hyperbolic segments $J[u_i,v_i]\subset\B^2$, $i=1,...,n$, are of the same hyperbolic length and have a common hyperbolic midpoint $q$, all their Euclidean counterparts $[u_i,v_i]$ intersect at the same point.
\end{theorem}
\begin{proof}
Choose distinct points $u_1,v_1\in\B^2$ that are non-collinear with the origin. Let $q$ be their hyperbolic midpoint, $R=\rho_{\B^2}(u_1,v_1)$ and $k=L(0,q)\cap L(u_1,v_1)$. Fix $j,h$ as in Lemma \ref{lem_rhoball}. Now, $u_1,v_1\in S^1(j,h)$, $J[u_1,v_1]\perp S^1(j,h)$ and $u_1,v_1,j$ are non-collinear. It follows from Theorem \ref{gendis_result} that the hyperbolic midpoint of $J[j,k]$ is $[j,k]\cap J[u_1,v_1]$. Since $0,j,q$ are collinear and $k\in L(0,q)$, $[j,k]\cap J[u_1,v_1]=L(0,q)\cap J[u_1,v_1]=q$. Thus, $q$ is the hyperbolic midpoint of $J[j,k]$. It follows that $k$ only depends on $q$ and $j$, so the intersection point $L(0,q)\cap L(u_i,v_i)$ must be same for all indexes $i$, as can be seen in Figure \ref{fig3}. If $u_i,v_i$ are collinear with the origin for some index $i$, then $k\in[u_i,v_i]$ trivially. Thus, the theorem follows.
\end{proof}

\begin{corollary}\label{cor_kpoint}
For all $x,y\in\B^2$, there is a point 
\begin{align*}
k=L(x,y)\cap L(x_2,y_2)\cap L(x_3,y_3).    
\end{align*}
\end{corollary}
\begin{proof}
Follows from Theorem \ref{thm_uvcounter}.
\end{proof}

\begin{theorem}\label{thm_theta_xy}
For all $x,y\in\B^2$ that are non-collinear with the origin and have a hyperbolic midpoint $q$, the distance $|x-y|$ is decreasing with respect to the smaller angle between $L(x,y)$ and $L(0,q)$.
\end{theorem}
\begin{proof}
Consider a hyperbolic circle $S^1_\rho(q,R)$, where $R=\rho_{\B^2}(x,q)$, and let $S^1(j,h)$ be the corresponding Euclidean circle. By Lemma \ref{lem_rhoball}, we see that the points $0,j,q$ are collinear. Fix $k$ as in Corollary \ref{cor_kpoint} and let $u=|j-k|$. Denote $\theta=\measuredangle(L(x,y),L(0,q))=\measuredangle(L(x,y),L(j,k))\in[0,\pi\slash2]$. Clearly, the distance $u$ does not depend on the angle $\theta$. It follows that $|x-y|=2\sqrt{h^2-u^2\sin^2(\theta)}$ is decreasing with respect to $\theta$.
\end{proof}

\begin{corollary}\label{cor_xygrows}
For all $x,y\in\B^2$, $|x_2-y_2|\leq|x-y|\leq|x_3-y_3|$.
\end{corollary}
\begin{proof}
Follows from Theorem \ref{thm_theta_xy}.
\end{proof}

\begin{corollary}
For all $x,y\in\B^2$,
\begin{align*}
|x-y|\leq\frac{2(1-|q|^2)t}{1-|q|^2t^2}\leq2{\rm th}(\rho_{\B^2}(x,y)\slash4).    
\end{align*}
where $q$ is the hyperbolic midpoint of $J[x,y]$, and $t={\rm th}(\rho_{\B^2}(x,y)\slash4)$.
\end{corollary}
\begin{proof}
By fixing $h$ as in Lemma \ref{lem_rhoball}, we will have
\begin{align*}
|x_3-y_3|=2h=\frac{2(1-|q|^2)t}{1-|q|^2t^2}
\leq2t
=2{\rm th}(\rho_{\B^2}(x,y)\slash4),
\end{align*}
so the result follows from Corollary \ref{cor_xygrows}.
\end{proof}

\begin{remark}
The inequality $|x-y|\leq2{\rm th}(\rho_{\B^2}(x,y)\slash4)$ can be also found in \cite[(4.25), p. 57]{hkvbook}.
\end{remark}

\begin{theorem}\label{thm_x2y2}
For all $x,y\in\B^2$, $s_{\B^2}(x,y)\geq s_{\B^2}(x_2,y_2)$.
\end{theorem}
\begin{proof}
Let $q$ be the hyperbolic midpoint of $J[x,y]$ and $R=\rho_{\B^2}(x,q)$. If $q=0$, $s_{\B^2}(x,y)=s_{\B^2}(x_2,y_2)$ holds trivially. Thus, choose $x,y\in\B^2$ so that $0<q<1$. Now, either the infimum $\inf_{z_2\in S^1}(|x_2-z_2|+|z_2-y_2|)$ is given by one point $z_2\in S^1$ or two points on $S^1$. 

If there is only one point giving the infimum $\inf_{z_2\in S^1}(|x_2-z_2|+|z_2-y_2|)$, it must be $z_2=1$ by Remark \ref{rmk_onlyonez}. Let $\overline{B}^2(j,h)=\overline{B}_\rho^2(q,R)$ like in Lemma \ref{lem_rhoball}, and fix $k$ as in Corollary \ref{cor_kpoint}. By symmetry, we can assume that $\measuredangle 1KX=\mu\in[0,\pi\slash2]$. Note that, if $\mu=\pi\slash2$, then $x=x_2$ and $y=y_2$. Now, it follows from Theorem \ref{fujimuras_result} that
\begin{align*}
s_{\B^2}(x,y)\geq\frac{|x-y|}{|x-1|+|1-y|}\geq\frac{|x_2-y_2|}{|x_2-1|+|1-y_2|}=s_{\B^2}(x_2,y_2). 
\end{align*}

Suppose now that there are two possible points $z_2\in S^1$ for $\inf_{z_2\in S^1}(|x_2-z_2|+|z_2-y_2|)$. By symmetry, let $\text{Im}(x_2)>0$ and $0\leq\arg(x)\leq\arg(x_2)$. Fix $z_2$ so that $\text{Im}(z_2)>0$ and $\measuredangle OZ_2X_2=\measuredangle Y_2Z_2O$, where $o$ is the origin. By Theorem \ref{find_zinB}, this point $z_2$ gives the infimum $\inf_{z_2\in S^1}(|x_2-z_2|+|z_2-y_2|)$. Denote yet $\psi=\measuredangle Y_2X_2Z_2$, which is clearly an obtuse angle. 

By Corollary \ref{cor_xygrows}, we can fix $y'\in[x,y]$ so that $|x-y'|=|x_2-y_2|$. Let $z\in S^1$ with $\text{Im}(z)<\text{Im}(x)$ so that $\measuredangle Y'XZ=\psi$. Clearly, $|x-z|\leq|x_2-y_2|$. By Proposition \ref{prop_linediameter}, it follows that
\begin{align*}
s_{\B^2}(x,y)
&\geq s_{\B^2}(x,y')
\geq\frac{|x-y'|}{|x-z|+|z-y'|}\\
&=\frac{|x-y'|}{|x-z|+\sqrt{|x-y'|^2+|x-z|^2-2|x-y'||x-z|\cos(\psi)}}\\
&\geq\frac{|x_2-y_2|}{|x_2-z_2|+\sqrt{|x_2-y_2|^2+|x_2-z_2|^2-2|x_2-y_2||x_2-z_2|\cos(\psi)}}\\
&=\frac{|x_2-y_2|}{|x_2-z_2|+|z_2-y_2|}
=s_{\B^2}(x_2,y_2)
\end{align*}
\end{proof}

\begin{theorem}\label{thm_x3y3}
For all $x,y\in\B^2$,
\begin{align*}
s_{\B^2}(x,y)\leq s_{\B^2}(x_3,y_3)=\frac{(1+|q|)t}{1+|q|t^2}, 
\end{align*}
where $q$ is the hyperbolic midpoint of $J[x,y]$, and $t={\rm th}(\rho_{\B^2}(x,y)\slash4)$.
\end{theorem}
\begin{proof}
Let $q$ be the hyperbolic midpoint of $J[x,y]$. Fix then $R=\rho_{\B^2}(x,q)$ and $j,h,t$ as in Lemma \ref{lem_rhoball}. Now, $\overline{B}^2(j,h)=\overline{B}_\rho^2(q,R)$ and $t={\rm th}(\rho_{\B^2}(x,y)\slash4)$. By Corollary \ref{cor_diam_diskinB},
\begin{align*}
s_{\B^2}(x,y)&\leq s_{\B^2}(\overline{B}_\rho^2(q,R))
=s_{\B^2}(\overline{B}^2(j,h))
=\frac{h}{1-|j|}
=\frac{(1-|q|^2)t}{1-|q|^2t^2-|q|(1-t^2)}\\
&=\frac{(1-|q|^2)t}{1-|q|+|q|t^2-|q|^2t^2}
=\frac{(1-|q|)(1+|q|)t}{(1-|q|)(1+|q|t^2)}
=\frac{(1+|q|)t}{1+|q|t^2}.
\end{align*}
Since $|j|=|x_3+y_3|\slash2$ and $h=|x_3-y_3|\slash2$, by Lemma \ref{lem_smetricinB_collinear},
\begin{align*}
s_{\B^2}(x_3,y_3)&=\frac{|x_3-y_3|}{2-|x_3+y_3|}=\frac{h}{1-|j|},
\end{align*}
so the theorem follows.
\end{proof}

According to numerous computer tests, the following result holds.

\begin{conjecture}\label{conj_hmr}
For all $x,y\in\B^2$,\newline
1. $s_{\B^2}(x_2,y_2)\geq s_{\B^2}(x_0,y_0)$,\newline
2. $s_{\B^2}(x_3,y_3)\leq s_{\B^2}(x_1,y_1)$,\newline
3. $s_{\B^2}(x_3,y_3)\leq p_{\B^2}(x,y)$,\newline
where the points $x_i,y_i$, $i=0,...,3$, are as in Definitions \ref{def_emr} and \ref{def_hmr}. 
\end{conjecture}

Thus, by this conjecture, the hyperbolic midpoint rotation gives sharper estimations for $s_{\B^2}(x,y)$ than the Euclidean midpoint rotation or the point pair function.

\section{H\"older continuity}\label{section_holder}

In this section, we show how finding better upper bounds for the triangular ratio metric in the unit disk is useful when studying quasiconformal mappings. The behaviour of the distance between two points $x,y\in\B^n$ under a $K$-quasiconformal homeomorphism $f:\B^n\to\B^n=f(\B^n)$ has been studied earlier in numerous works, for instance, see \cite[Thm 16.14, p. 304]{hkvbook}. Our next theorem illustrates how finding a good upper limit for the value of the triangular ratio metric can give new information regarding this question.

\begin{theorem}\label{thm_s_formori}
If $f:\B^2\to\B^2=f(\B^2)$ is a $K$-quasiconformal map, the inequality
\begin{align*}
|f(x)-f(y)|\leq2^{3-1\slash K}\left(\frac{s_{\B^2}(x,y)}{1+s_{\B^2}(x,y)^2}\right)^{1\slash K},
\end{align*}
holds for all $x,y\in\B^2$.
\end{theorem}
\begin{proof}
Define a homeomorphism $\varphi_K:[0,1]\to[0,1]$ as in \cite[(9.13), p. 167]{hkvbook} for $K>0$. By \cite[Thm 9.32(1), p. 167]{hkvbook} and \cite[(9.6), p. 158]{hkvbook},
\begin{align}\label{varphi_ineq}
\varphi_K(r)\leq4^{1-1\slash K}r^{1\slash K}
=4^{1-1\slash(2K)}\left(\frac{r}{2}\right)^{1\slash K},
\end{align}
where $0\leq r\leq1$ and $K\geq1$. Let $f$ be as above, $x,y\in\B^2$ and $t=\text{th}(\rho_{\B^2}(x,y)\slash4)$. By Theorem \ref{rhojsp_inB}, the Schwarz lemma (see \cite[Thm 16.2, p. 300]{hkvbook}) and the inequality \eqref{varphi_ineq},
\begin{align*}
s_{\B^2}(f(x),f(y))&\leq\text{th}\frac{\rho_{\B^2}(f(x),f(y))}{2}\leq\varphi_K\left(\text{th}\frac{\rho_{\B^2}(x,y)}{2}\right)
=\varphi_K\left(\frac{2t}{1+t^2}\right)\\
&\leq4^{1-1\slash(2K)}\left(\frac{t}{1+t^2}\right)^{1\slash K}
\leq4^{1-1\slash(2K)}\left(\frac{s_{\B^2}(x,y)}{1+s_{\B^2}(x,y)^2}\right)^{1\slash K}.
\end{align*}
By \cite[Lemma 11.12, p. 201; Prop. 11.15, p. 202]{hkvbook}, it follows from the inequality above that
\begin{align*}
|f(x)-f(y)|\leq2s_{\B^2}(f(x),f(y))
\leq2^{3-1\slash K}\left(\frac{s_{\B^2}(x,y)}{1+s_{\B^2}(x,y)^2}\right)^{1\slash K},
\end{align*}
which proves the theorem.
\end{proof}

Thus, as we see from Theorem \ref{thm_s_formori}, finding a suitable upper bound for the value of $s_{\B^2}(x,y)$ can help us estimating the distance of the points $x,y$ under the $K$-quasiconformal mapping $f$.

\begin{corollary}\label{cor_morip}
If $f$ is as in Theorem \ref{thm_s_formori}, the inequality
\begin{align*}
|f(x)-f(y)|
\leq2^{3-2\slash K}\left(\frac{\sqrt{|x-y|^2+4(1-|x|)(1-|y|)}|x-y|}{|x-y|^2+2(1-|x|)(1-|y|)}\right)^{1\slash K}.
\end{align*}
holds for all $x,y\in\B^2$.
\end{corollary}
\begin{proof}
It follows from Theorems \ref{thm_s_formori} and \ref{rhojsp_inB} that
\begin{align*}
|f(x)-f(y)|
&\leq2^{3-1\slash K}\left(\frac{s_{\B^2}(x,y)}{1+s_{\B^2}(x,y)^2}\right)^{1\slash K}
\leq2^{3-1\slash K}\left(\frac{p_{\B^2}(x,y)}{1+p_{\B^2}(x,y)^2}\right)^{1\slash K}\\
&=2^{3-1\slash K}\left(\frac{\sqrt{|x-y|^2+4(1-|x|)(1-|y|)}|x-y|}{2|x-y|^2+4(1-|x|)(1-|y|)}\right)^{1\slash K}\\
&=2^{3-2\slash K}\left(\frac{\sqrt{|x-y|^2+4(1-|x|)(1-|y|)}|x-y|}{|x-y|^2+2(1-|x|)(1-|y|)}\right)^{1\slash K}.
\end{align*}
\end{proof}

\begin{corollary}\label{cor_mori1}
If $f$ is as in Theorem \ref{thm_s_formori}, the inequality
\begin{align*}
|f(x)-f(y)|
\leq2^{3-2\slash K}\left(\frac{(2-|x+y|)|x-y|}{2-2|x+y|+|x|^2+|y|^2}\right)^{1\slash K}.
\end{align*}
holds for all $x,y\in\B^2$.
\end{corollary}
\begin{proof}
Follows from Theorem \ref{thm_s_formori} and Lemma \ref{lem_smetricinB_collinear}, and the fact that $|x+y|^2+|x-y|^2=2|x|^2+2|y|^2$.
\end{proof}

\begin{corollary}\label{cor_morihyprot}
If $f$ is as in Theorem \ref{thm_s_formori}, all $x,y\in\B^2$ fulfill
\begin{align*}
|f(x)-f(y)|
\leq2^{3-1\slash K}\left(\frac{(1+|q|)(1+|q|t^2)t}{(1+|q|t^2)^2+(1+|q|)^2t^2}\right)^{1\slash K},
\end{align*}
where $q$ is the hyperbolic midpoint of $J[x,y]$, and $t={\rm th}(\rho_{\B^2}(x,y)\slash4)$.
\end{corollary}
\begin{proof}
Follows from Theorems \ref{thm_s_formori} and \ref{thm_x3y3}.
\end{proof}

\begin{remark}
Neither of Corollaries \ref{cor_mori1} and \ref{cor_morip} is better than the other for all points $x,y\in\B^2$. For $x=0.3$ and $y=0.3i$, the limit in Corollary \ref{cor_mori1} is sharper than the one in Corollary \ref{cor_morip} and, for $x=0.9$ and $y=0.9i$, the opposite holds. However, according to numerical tests related to Conjecture \ref{conj_hmr}, the result in Corollary \ref{cor_morihyprot} is always better than the ones in Corollaries \ref{cor_mori1} and \ref{cor_morip}.   
\end{remark}

By restricting how the point pair $x,y$ is chosen from $\B^2$, we can find yet better estimates.

\begin{corollary}\label{cor_moriWithR}
If $f$ is as in Theorem \ref{thm_s_formori}, the inequality
\begin{align*}
|f(x)-f(y)|
\leq2^{3-2\slash K}\left(\frac{|x-y|}{1-r}\right)^{1\slash K}.
\end{align*}
holds for all $x,y\in\B^2$ such that $|x+y|\slash2\leq r$.
\end{corollary}
\begin{proof}
Now,
\begin{align*}
\frac{2-|x+y|}{2-2|x+y|+|x|^2+|y|^2}
\leq\frac{2-|x+y|}{2-2|x+y|+|x+y|^2\slash2}
=\frac{1}{1-|x+y|\slash2}
\leq\frac{1}{1-r}
\end{align*}
so the result follows from Corollary \ref{cor_mori1}.
\end{proof}

\begin{corollary}
For all $x,y\in\B^2$ such that $|x+y|\leq1$,
\begin{align*}
|f(x)-f(y)|\leq2^{3-1\slash K}|x-y|^{1\slash K},   
\end{align*}
where $f$ is as in Theorem \ref{thm_s_formori}.
\end{corollary}
\begin{proof}
Follows from Corollary \ref{cor_moriWithR}.
\end{proof}

\begin{remark}
The proof of Theorem \ref{thm_s_formori} is based on the Schwarz lemma of quasiregular mappings \cite[Thm 16.2, p. 300]{hkvbook} and therefore the results of this section hold also for quasiregular mappings with minor modifications.
\end{remark}

\def\cprime{$'$} \def\cprime{$'$} \def\cprime{$'$}
\providecommand{\bysame}{\leavevmode\hbox to3em{\hrulefill}\thinspace}
\providecommand{\MR}{\relax\ifhmode\unskip\space\fi MR }
\providecommand{\MRhref}[2]{%
  \href{http://www.ams.org/mathscinet-getitem?mr=#1}{#2}
}
\providecommand{\href}[2]{#2}


\begin{thebibliography}{10}

\bibitem{avv}{\sc
G. Anderson, M. Vamanamurthy and M. Vuorinen,}
\emph{Conformal Invariants, Inequalities, and Quasiconformal Maps.} Wiley-Interscience, 1997.

\bibitem{bm}{\sc 
A.F. Beardon and D. Minda,} 
The hyperbolic metric and geometric function theory, \emph{Proc. International Workshop on  Quasiconformal Mappings and their Applications (IWQCMA05)}, eds. S. Ponnusamy, T. Sugawa and M. Vuorinen (2006), 9-56.

\bibitem{chkv}{\sc
J. Chen, P. Hariri, R. Kl\'en and M. Vuorinen,}
Lipschitz conditions, triangular ratio metric, and quasiconformal maps.
\emph{Ann. Acad. Sci. Fenn. Math., 40} (2015), 683-709.

\bibitem{fhmv}{\sc
M. Fujimura, P. Hariri, M. Mocanu and M. Vuorinen,}
The Ptolemy–Alhazen Problem and Spherical Mirror
Reflection.
\emph{Comput. Methods and Funct. Theory, 19} (2019), 135-155.

\bibitem{fmv}{\sc 
M. Fujimura, M. Mocanu and M. Vuorinen,} 
Barrlund's distance function and quasiconformal
maps, \emph{Complex Var. Elliptic Equ.} (2020), 1-31.

\bibitem{gh}{\sc
F. W. Gehring and K. Hag,}
\emph{The ubiquitous quasidisk}, vol. 184 of Mathematical Surveys and Monographs. American Mathematical Society, Providence, RI, 2012. With contributions by Ole Jacob Broch.

\bibitem{gmp}{\sc
F. W. Gehring, G. J. Martin and B. P. Palka,}
\emph{An introduction to the theory of higher dimensional quasiconformal mappings}, vol. 216 of Mathematical Surveys and Monographs. American Mathematical Society, Providence, RI, 2017.

\bibitem{gp}{\sc
F.W. Gehring and B. P. Palka,}
Quasiconformally homogeneous domains, \emph{J. Analyse Math. 30} (1976), 172--199.

\bibitem{hkvbook}{\sc
P. Hariri, R. Kl\'en and M. Vuorinen,}
\emph{Conformally Invariant Metrics and Quasiconformal Mappings.}
Springer, 2020.

\bibitem{hkvz}{\sc
P. Hariri, R. Kl\'en, M. Vuorinen and X. Zhang,}
Some Remarks on the Cassinian Metric.
\emph{Publ. Math. Debrecen, 90}, 3-4 (2017), 269-285.

\bibitem{hvz}{\sc
P. Hariri, M. Vuorinen and X. Zhang,}
Inequalities and Bilipschitz Conditions for Triangular Ratio Metric.
\emph{Rocky Mountain J. Math., 47}, 4 (2017), 1121-1148.

\bibitem{h}{\sc  
P. H\"ast\"o,} 
A new weighted metric, the relative metric I. \emph{J. Math. Anal. Appl. 274} (2002), 38-58.

\bibitem{hkst}{\sc  
J. Heinonen, P. Koskela, N. Shanmugalingam and J. T. Tyson,} 
\emph{Sobolev spaces on metric measure spaces}, vol. 27 of New Mathematical Monographs. Cambridge University Press, Cambridge, 2015. An approach based on upper gradients.

\bibitem{p14}{\sc  
A. Papadopoulos,} 
\emph{Metric spaces, convexity and non-positive curvature}, Second ed., vol. 6 of IRMA Lectures in Mathematics and Theoretical Physics. European Mathematical Society (EMS), Z\"urich, 2014.

\bibitem{v99}{\sc
J. V\"ais\"al\"a,}
The free quasiworld. Freely quasiconformal and related maps in Banach spaces. In Quasiconformal geometry and dynamics (Lublin, 1996), vol. 48 of Banach Center Publ. \emph{Polish Acad. Sci. Inst. Math.}, Warsaw, 1999, 55-118.

\bibitem{vw}{\sc
M. Vuorinen and G. Wang,}
Hyperbolic Lambert Quadrilaterals and Quasiconformal mappings.
\emph{Ann. Acad. Sci. Fenn. Math., 38} (2013), 433-453.

\bibitem{wvz}{\sc
G. Wang, M. Vuorinen and X. Zhang,}
On Cyclic Quadrilaterals in Euclidean and Hyperbolic Geometries. arXiv:1908.10389.

\end{thebibliography}
\end{document}